\documentclass[11pt,a4paper,twoside,reqno]{amsart}
\usepackage{amsmath,amsfonts,amsthm,amsopn,color,amssymb,enumitem}
\usepackage{palatino}
\usepackage{graphicx}
\usepackage[colorlinks=true]{hyperref}
\hypersetup{urlcolor=blue, citecolor=red, linkcolor=blue}

\usepackage{relsize}
\usepackage{esint}

\newcommand{\e}{\varepsilon}

\newcommand{\R}{\mathbb{R}}
\newcommand{\N}{\mathbb{N}}

\newcommand{\de}{\partial}
\newcommand{\weakto}{\rightharpoonup}

\renewcommand{\S}{\Sigma}
\newcommand{\deS}{{\de\Sigma}}
\newcommand{\Sph}{{\mathbb{S}^1}}

\newcommand{\D}{{\mathbb{D}^2}}
\newcommand{\X}{\mathbb{X}}
\newcommand{\hsp}{\hspace{0.2cm}}

\newcommand*{\modgrad}[1]{\left\vert \nabla #1\right\vert^2}

\newcommand{\beq }{\begin{equation}}
\newcommand{\eeq }{\end{equation}}

\newtheorem{theorem}{Theorem}[section]
\newtheorem{lemma}[theorem]{Lemma}

\newtheorem{proposition}[theorem]{Proposition}
\newtheorem{remark}[theorem]{Remark}
\newtheorem{corollary}[theorem]{Corollary}

\title[Prescribing Gaussian and geodesic curvature on disks]
{Prescribing Gaussian and geodesic curvature on disks}

\author{Sergio Cruz-Bl\'{a}zquez}
\address{Sergio Cruz-Bl\'{a}zquez, Scuola Normale Superiore, Piazza dei Cavalieri 7, 56126 Pisa, Italy.}
\email{sergio.cruzblazquez@sns.it }

\author{David Ruiz}
\address{David Ruiz, Departamento de An\'{a}lisis Matem\'{a}tico, Universidad de Granada, Campus Fuentenueva, 18071 Granada, Spain.}
\email{daruiz@ugr.es}
\thanks{D. R. have been supported by the Feder-Mineco Grant MTM2015-68210-P and by J. Andalucia (FQM116).} 
\thanks{S. C-B. is supported by the Marie Sklodowska-Curie fellowship of the Istituto Nazionale di Alta Matematica 713485.}

\keywords{Prescribed Gaussian curvature problem, variational methods, Moser-Trudinger inequality.}

\subjclass[2010]{35J20, 35R01, 53A30.}

\begin{document}

\maketitle

\begin{abstract} In this paper we consider the problem of prescribing the Gaussian and geodesic curvature on a disk and its boundary, respectively, via a conformal change of the metric. This leads us to a Liouville-type equation with a nonlinear Neumann boundary condition. We address the question of existence by setting the problem in a variational framework which seems to be completely new in the literature. We are able to find minimizers under symmetry assumptions.
\end{abstract}

\section{Introduction}

The problem of prescribing the Gaussian curvature on a compact surface $\Sigma$ under a conformal change of the metric is a classical one, and dates back to \cite{Ber, KazWar}. Let us denote by $g$ the original metric, $\tilde{g}$ the new one and $e^u$ the conformal factor (that is, $\tilde{g} = e^u g$). This problem reduces to solving the problem
$$ - \Delta_g u + 2 K_g = 2 K_{\tilde{g}}e^{u},$$
where $K_g$, $K_{\tilde{g}}$ denote the curvature with respect to $g$ and $\tilde{g}$, respectively. The solvability of this equation has been studied for a long time, and it is not possible to give here a comprehensive list of references.  

If $\Sigma$ has a boundary, then boundary conditions are in order. Homogeneous Dirichlet and Neumann boundary conditions have already been considered in the literature. In this paper our aim is to prescribe not only the Gaussian curvature in $\S$, but also the geodesic curvature on $\partial \S$. In this case we are led with the boundary value problem:
\begin{equation} \label{gg}
\left\{\begin{array}{ll}
-\Delta_{g} u +2K_g = 2 K_{\tilde{g}} e^u  \qquad & \text{in $\S$,}\\
\frac{\partial u}{\partial n} +2h_g = 2h_{\tilde{g}}e^{u/2} \qquad  &\text{on $\partial\S $,}
\end{array}\right.
\end{equation}
 where $h_g, h_{\tilde{g}}$ are the geodesic curvatures of $\partial \S$ relative to $g$, $\tilde{g}$, respectively.

Some versions of this problem have been studied in the literature. The case $h_{\tilde{g}}=0$ has been treated  by A. Chang and P. Yang in \cite{chang1}. Moreover, the case $K_{\tilde{g}}=0$ has been treated in \cite{chang2, li-liu, liuwang}. There is also some progress in the blow-up analysis, see \cite{bao, francesca}, although a complete description of the phenomenon is still missing. 

The case of constants $K_{\tilde{g}}$, $h_{\tilde{g}}$ has also been considered. For instance, Brendle (\cite{brendle}) uses a parabolic flow to show that this problem admits always a solution for some constant curvatures. By using complex analysis techniques, explicit expressions for the solutions and the exact values of the constants are determined if $\Sigma$ is a disk or an annulus, see \cite{otro, Asun}. The case of the half-plane has also been studied, see \cite{li-zhu, mira-galvez, zhang}. However, the case in which both curvatures are not constant has not been much considered. In \cite{cherrier}, some partial existence results are given, but they include a Lagrange multiplier which is out of control. Moreover, a Kazdan-Warner type of obstruction to existence has been found in \cite{hamza}. In a forthcoming work, the case of $K<0$ in domains different from the disk is treated, and also a blow-up analysis is performed, see \cite{LS-M-R}. At present, as far as we know, those are the only works considering non-constant curvatures.

The higher dimensional analogue of this question (that is, prescribing scalar curvature of a manifold and mean curvature of the boundary) has been more studied. The case of zero scalar curvature and constant mean curvature is known as the Escobar problem, in strong analogy with the Yamabe problem. In this regard, see \cite{amb, dja, escobar Annals, escobar Indiana, felli, HanLi1, HanLi2, marques}, and the references therein.

Integrating \eqref{gg} and applying the Gauss-Bonnet Theorem, one obtains
\begin{equation} \label{GB}
\int_{\Sigma}K_{\tilde{g}} e^u +  \int_{\partial \Sigma}h_{\tilde{g}} e^{u/2}=  2\pi\chi(\Sigma).
\end{equation}

In this paper we shall consider the case in which $\chi(\Sigma)=1$. By the Uniformization Theorem, we can pass via a conformal map to a disk, obtaining $K_{\tilde{g}}=0$, $h_{\tilde{g}}=1$. Taking this into account we can consider the problem:

\begin{align}
\left\lbrace \begin{array}{ll}
-\Delta u = 2Ke^u & \mbox{in } \D, \\
\frac{\partial u}{\partial \eta} + 2 = 2he^{u/2} & \mbox{on } \Sph,
\end{array}\right. \label{Probl}
\end{align}
where now $K$, $h$ are the curvatures to be prescribed. 

Generally speaking, the case of a disk is specially challenging because of the non-compact action of the group of conformal maps of the disk, as happens in the Nirenberg problem for $\S = \mathbb{S}^2$. This issue has been only treated in \cite{chang2} for $K=0$ (see also \cite{francesca}). A blow-up analysis in this case for non-constant $K$, $h$ is yet to be done, and will be the target of further research. In this paper, as a first step in the understanding of the problem, we shall impose symmetry conditions on $K$, $h$ in order to rule out this phenomenon. This idea goes back to Moser (\cite{moser2}) for the Nirenberg problem.

Let $G$ be a symmetry group of $\D$ without fixed points on $\Sph$, that is, for each $x\in\Sph$ there exists $g\in G$ such that $g(x)\neq x$. We say that a function $f$ is $G-$symmetric if $f(x)=f(g(x))$ for all $g \in G$ and for all $x$ in the domain of $f$.

Our main results is the following.

\begin{theorem} \label{TheoremNonNegative} Let $K: \D \to \R$, $h: \Sph \to \R$ be H\"{o}lder continuous, nonnegative and $G-$symmetric functions, not both of them identically equal to $0$. Then problem \eqref{Probl} admits a solution.
\end{theorem}

We can also deal with changing sign curvatures $K$, $h$, as long as their negative part is small:

\begin{theorem} \label{TheoremPerturbed}Let $K_0: \D \to \R$, $h_0: \Sph \to \R$ be H\"{o}lder continuous, nonnegative and $G-$symmetric functions, none of them identically equal to $0$. Then there exists $\e >0$ such that problem \eqref{Probl} admits a solution for any H\"{o}lder continuous and $G-$symmetric functions $K$, $h$ with $\| K- K_0\|_{L^{\infty}} + \| h - h_0\|_{L^{\infty}} < \e$.
\end{theorem}

One of the main goals of this paper is to find an original variational setting to this problem, which we think is natural and could be of use in future research on the topic. Let us be more specific. We define the parameter $\rho := \int_\D K e^u = 2\pi - \int_{\Sph} he^{u/2}$. In order to fix ideas, let us assume that both $K$, $h$ are nonnegative functions; by \eqref{GB},  $0<\rho<2\pi$.

We shall show that \eqref{Probl} is equivalent to: 

\begin{align}
\left\lbrace \begin{array}{ll}
-\Delta u = 2\rho \frac{Ke^u}{\int_\D Ke^u} & \mbox{in } \D, \\
\frac{\partial u}{\partial \eta} + 2 = 2(2\pi-\rho)\frac{he^{u/2}}{\int_\Sph he^{u/2}} & \mbox{on } \Sph, \\ \frac{(2\pi-\rho)^2}{\rho}=\frac{\left(\int_\Sph he^{u/2}\right)^2}{\int_\D Ke^u} & \mbox{for } 0<\rho<2\pi.
\end{array}\right. \label{Probl3}
\end{align}

Observe that problem \eqref{Probl3} is now invariant under addition of constants to $u$, and $\rho$ is here an unknown. This formulation may seem rather artificial but it has the advantage of being related to the critical points of the energy functional:

\begin{align}
I(u,\rho) &= \frac{1}{2} \int_\D \vert \nabla u\vert^2 - 2\rho \log \int_\D Ke^u + 2\int_{\Sph} u - 4(2\pi-\rho)\log\int_{\Sph} he^{u/2} \nonumber\\ &+4(2\pi-\rho)\log(2\pi-\rho)+2\rho +2\rho\log\rho. \label{DefFunctional}
\end{align}

We highlight the fact that the functional above depends on the couple $(u, \rho)$, where $u \in H^1(\D)$ and $\rho \in (0, 2\pi)$. The form of this energy functional seems to be completely new in the related literature. 

If we freeze the variable $\rho$, the form of this functional is adequate for the use of Moser-Trudinger type inequalities (or Onofri-type inequalitites) which are already available also for boundary terms. Indeed, by interpolating these inequalities we will show that $I$ is bounded from below. We will gain coercivity in the $u$ variable by imposing symmetry, as first done by Moser in \cite{moser2}. Finally, we will need to exclude the possibility of obtaining minima at the endpoints $\rho=0$ or $\rho = 2 \pi$. Those limit cases correspond to the problem in which $K=0$ or $h=0$, respectively, so some study of these cases is needed. By energy estimates we can assure that the minimum is attained at $\rho \in (0,2\pi)$, concluding the proof.

If either $K$ or $h$ changes sign  the above approach fails. We shall also give in Theorem \ref{TheoremGeneral} a more general result; as a corollary, and making use of a compactness result for minima of the functional $I$, we will obtain the perturbation result stated in Theorem \ref{TheoremPerturbed} .

The rest of the paper is organized as follows. In Section 2 we set the notation and the variational formulation of the problem. After that, an analysis of the properties of the energy functional is performed by means of Moser-Trudinger type inequalities. Section 3 is devoted to the proof of Theorem \ref{TheoremNonNegative}, for which we first need to address the limiting cases $\rho=0$ and $\rho=2\pi$. A more general version is also given. Finally, the proof of Theorem \ref{TheoremPerturbed} is completed in Section ~\ref{sect4}.

\section{Variational Setting}
\setcounter{equation}{0}
\label{sect2}

\subsection{Notations}

Let us first set some notations. Given a set $A \subset X$ in a metric space, we denote:

$$(A)^r = \{x \in X: \mbox{dist}(x,A)<r\}.$$ 
Regarding the integrals, in this paper we shall consider only the Lebesgue measure and we drop the element of area or length, that is, we shall only write $\int_\D K e^u$ or $\int_{\Sph} h e^{u/2}$. We also use the symbol $\fint f$ to denote the mean value of $f$, that is, $$\fint_{\Sigma} f = \frac{1}{\vert \Sigma \vert}\int_\S f.$$
 
In our estimates we sometimes write $C$ to denote a positive constant, independent of the variables considered, that may change from line to line.

\subsection{Variational formulation}

As commented in the introduction, we will consider the functional $I$ given by \eqref{DefFunctional} and defined on the space
$$\X\times (0,2\pi) = \left\lbrace u\in H^1(\D): \int_\D Ke^u >0, \int_{\Sph} he^{u/2} >0 \right\rbrace\times (0,2\pi).$$
With the purpose of clarifying the notation, for a fixed $\rho \in (0,2\pi)$ we call $I_{\rho}$ to the functional $u\to I(u,\rho)$ defined for every $u\in\X$. We should notice that the functionals $I_\rho$ are invariant under the addition of constants.

\begin{lemma} \label{notempty} $\X$ is nonempty if and only if $K$ and $h$ are positive somewhere. 
\end{lemma}

\begin{proof}
We reduce ourselves to prove that if $K$ and $h$ are positive somewhere then $\X$ is nonempty, as the reciprocal is immediate. As $K$ is continuous and there exists $x_0\in\mbox{Int}(\D)$ such that $K(x_0)>0$, then there exists $r>0$ such that $(\{x_0\})^r\cap \Sph = \emptyset$ and $K(x)>0$ for all $x\in (\{x_0\})^r$. 

Moreover, we know that there exists $x_1\in\Sph$ satisfying $h(x_1)>0$, and again by continuity we get $s>0$ such that $h(x)>0$ for all $x\in (\{x_1\})^s\cap\Sph$. It is not restrictive to assume $(\{x_0\})^r\cap (\{x_1\})^s=\emptyset$. We call $\Omega_0^r := (\{x_0\})^r$ and $\Omega_1^s := (\{x_1\})^s$ and consider a cutoff function $\varphi \in H^1(\D)$ satisfying
\begin{align*}
\varphi(x)=\left\lbrace\begin{array}{ll} a & \mbox{if } x\in\Omega_0^{r/2}, \\ b & \mbox{if } x\in\Omega_1^{s/2}, \\ 0 & \mbox{if } \D\backslash \left(\Omega_0^r \cup \Omega_1^s\right),\end{array} \right.
\end{align*}
where $a$ and $b$ are real constants to determine. We see that:
\begin{align*}
\int_\Sph he^{\varphi/2} &= \int_{\Omega_1^{s/2}\cap\partial\D}he^{\varphi/2}+\int_{\left(\Omega_1^s\backslash\Omega_1^{s/2}\right)\cap\partial\D}he^{\varphi/2}+ \int_{\partial\D\backslash \Omega_1^s}he^{\varphi/2} \\
&\geq e^{b/2}\int_{\Omega_1^{s/2}\cap\partial\D}h+ \int_{\partial\D\backslash \Omega_1^{s/2}}h = C_1e^{b/2} + C,
\end{align*}
being $C_1>0$ and $C\in\R$. We can choose $b$ large enough so that $$\int_\Sph he^{\varphi/2}>0.$$
Furthermore,
\begin{align*}
\int_\D Ke^\varphi &= \int_{\Omega_0^{r/2}}Ke^\varphi + \int_{\Omega_1^{s/2}}Ke^\varphi + \int_{\D\backslash\left(\Omega_0^r\cup\Omega_1^r\right)}Ke^\varphi + \int_{\Omega_1^s\backslash \Omega_1^{s/2}}Ke^\varphi \\ &+ \int_{\Omega_0^r\backslash \Omega_0^{r/2}}Ke^\varphi \geq e^a \int_{\Omega_0^{r/2}}K - e^b C_2 \Vert K \Vert_\infty + C = C_1'e^a - C_2'e^b + C.
\end{align*}
So we can also set $a$ big enough so that $$\int_\D Ke^\varphi >0.$$
\end{proof}

Let us point out that the Euler-Lagrange equation of $I$ is given by \eqref{Probl3}, which is a reformulation of \eqref{Probl}, in view of next Lemma:
\begin{lemma} Problems \eqref{Probl} and \eqref{Probl3} are equivalent.
\end{lemma}
\begin{proof}
In order to check that every solution of \eqref{Probl} is a solution of \eqref{Probl3} we just need to take $\rho = \int_\D Ke^u = 2\pi-\int_\Sph he^{u/2} >0.$ Reciprocally, if $u\in \X$ solves (\ref{Probl3}), applying the invariance under addition of constant of that problem we have, for any $C\in\R:$
\begin{align*}
-\Delta (u+C) &= 2\rho \frac{Ke^{u+C}}{e^C\int_\D Ke^u}\, , \\
\frac{\partial(u+C)}{\partial \eta}+2 &= 2(2\pi-\rho)\frac{he^{u/2}}{e^{\frac{C}{2}}\int_\Sph he^{u/2}}\, .
\end{align*}
If we want $u+C$ to solve \eqref{Probl}, we need $C\in\R$ such that $$e^C = \frac{\rho}{\int_\D K e^u}\, ,\ \ e^{\frac{C}{2}}=\frac{(2\pi-\rho)}{\int_\Sph he^{u/2}}\,.$$
The third equation of \eqref{Probl3} tells us that both conditions are actually the same. Thus, it is enough to choose $C=\log\rho - \log \int_\D Ke^u$.
\end{proof}

\subsection{Moser-Trudinger inequalities}
The Moser-Trudinger inequalities (see \cite{chang1,moser1,moser2,trudinger1}) and their variations are useful tools to deal with the non-linear terms of exponential type which appear in our functional. In particular we are interested in weaker versions of these inequalities, also called Onofri type inequalities.

\begin{theorem}\label{MosTrudWeak} Let $\S$ be a compact surface with $C^1$ boundary. Then there exists a constant $C\in\R$, depending only on $\S$, such that
\begin{equation}
\log \int_\S e^u \leq \frac{1}{16\pi} \int_\S \vert \nabla u \vert ^2 + C \hsp \hsp \forall u\in H^1_0(\S), \label{equa3}
\end{equation}
and
\begin{equation}
\log \int_\S e^u \leq \frac{1}{8\pi} \int_\S \vert \nabla u\vert^2 + \fint_{\S}u + C \hsp \hsp \forall u\in H^1(\S). \label{equa4}
\end{equation}
\end{theorem}

The first inequality is classical, whereas the second is given in \cite[Proposition 2.3 and subsequent corollary]{chang1}. In both cases the constant is optimal.

%It is important to observe that the constant in \eqref{equa4} is doubled with respect to the one in \eqref{equa3}, since we can center a bubble in a point of $\de\S$ so that its Dirichlet energy is divided by approximately two. This phenomena does not appear in \eqref{equa4} because of the boundary condition.

In order to address the non-linear boundary terms of the functional $I$, we will use an analogous version of Theorem \ref{MosTrudWeak} for the boundary of a compact surface that can be found in \cite{li-liu}, for instance.

\begin{proposition}\label{MosTrudBoundWeak} Let $\S$ be a compact surface with $C^1$ boundary. Then there exists a constant $C>0$, depending only on $\S$, such that
	$$\log \int_{\de\S} e^u \leq \frac{1}{4\pi} \int_\S \vert \nabla u\vert^2 + \fint_{\de\S} u + C, \hsp \hsp \forall u\in H^1(\Sigma).$$
\end{proposition}

In the case of the disk, the above inequality is the so-called Lebedev-Milin inequality (with $C=0$, see for instance \cite[equation (4')]{osgood}).

By interpolating the previous inequalities we will obtain a lower bound for the functional $I$. First, we notice that inequality \eqref{equa4} can be manipulated so that the mean value of $u$ in $\de\S$ replaces the mean in $\S$.
\begin{corollary}\label{WeakModified}Let $\S$ be a compact surface with $C^1$ boundary. There exists a constant $C\in\R$, only depending on $\S$, such that $$\log \int_\S e^v \leq \frac{1}{8\pi} \int_\S \vert \nabla v \vert ^2 +\fint_\deS v + C\hsp \hsp \forall v\in H^1(\Sigma).$$
\end{corollary}

\begin{proof} We consider the problem 
	\begin{equation}
	\left\lbrace \begin{array}{ll} -\Delta w = \frac{-4\pi}{\vert \Sigma \vert} & \mbox{in } \Sigma, \\[0.15cm]
	\frac{\partial w}{\partial \eta} = \frac{4\pi}{\vert \deS\vert} & \mbox{on } \partial \Sigma. \end{array}\right. \label{equa5}
	\end{equation}
	Let us point out that \eqref{equa5} is solvable in $H^1(\Sigma)$ because $\int_\deS \frac{4\pi}{\vert \deS\vert} = -\int_\S \frac{4\pi}{\vert \S\vert} = 4\pi$. We fix a solution $w$ of \eqref{equa5} and apply \eqref{equa4} to $v+w$, obtaining:
	$$\log \int_\S e^v \leq \frac{1}{8\pi}\int_\S \vert \nabla v \vert ^2 + \frac{1}{4\pi} \int_\deS \frac{\partial w}{\partial \eta} v - \frac{1}{4\pi}\int_\S (\Delta w)v + \fint_{\Sigma}v + C.$$
	Finally, we use that $w$ solves \eqref{equa5}:
	\begin{align*}
	\log \int_\S e^v &\leq \frac{1}{8\pi} \int_\S \vert \nabla v \vert ^2 + \int_\deS v-\fint_{\Sigma} v + \fint_{\Sigma} v + C \\ =& \frac{1}{8\pi} \int_\S \vert \nabla v \vert ^2 + \int_\deS v+ C.
	\end{align*}
\end{proof}
In a similar way one can obtain a modified version of Proposition \ref{MosTrudBoundWeak} in which the mean value of $u$ on $\S$ substitutes the mean on $\partial\S$. 
\begin{corollary}\label{WeakModified2}Let $\Sigma$ be a compact surface with $C^1$ boundary. There exists $C\in\R$, only depending on $\S$, such that
\begin{align*}
\log \int_\deS e^u \leq \frac{1}{4\pi} \int_\S \modgrad{u} + \fint_{\Sigma} u + C \hsp \forall u\in H^1(\Sigma) .
\end{align*} 
\end{corollary}

The combined use of the inequality \eqref{equa5} and Corollary \ref{WeakModified} allows us to prove that $I$ is bounded from below in $H^1(\D)$. 

\begin{proposition}\label{LowerBound}There exists a constant $C\in\R$ such that $I_\rho(u)\geq C$ for every $u\in \X$ and every $\rho\in\:[0,2\pi]$.
\end{proposition}
\begin{proof}
Let us define $f:(0,2\pi)\to\R$ as the correction term in \eqref{DefFunctional}, that is
\begin{align*}
f(\rho)=4(2\pi-\rho)\log(&2\pi-\rho)+2\rho +2\rho\log\rho.
\end{align*}
It is clear that $$\lim_{\rho\to 0} f(\rho) = 8\pi \log(2\pi),\hsp \lim_{\rho\to 2\pi}f(\rho)= 4\pi + 4\pi\log(2\pi).$$
Then, $f$ can be continuously extendended to the compact $[0,2\pi]$. Thus, there exists a constant $M>0$ such that $\vert f(\rho)\vert \leq M$ for all $\rho\in$ $[0,2\pi]$. Moreover, since $K$ and $h$ are continuous, there exist $M_1, M_2 \in \R$ such that $$\log\int_\D K e^u \leq \log \int_\D e^u + C, \hsp \log\int_\Sph h e^{u/2}\leq \int_\Sph e^{u/2}+C.$$
Then, for every $a,b\in \R$:
\begin{align*}
I_\rho(u)&\geq \frac{1}{2}\int_\D\vert\nabla u\vert^2 -2\rho \log \int_\D e^u -4(2\pi-\rho)\log \int_\Sph e^{u/2} + 2\int_\Sph u + C \\
&=\frac{8\pi - 2a -b}{16\pi}\int_\D \vert\nabla u\vert^2 + \frac{a}{8\pi}\int_\D \vert\nabla u\vert^2 +\frac{b}{16\pi}\int_\D\vert \nabla u \vert ^2  \\&- 2\rho \log \int_\D e^u - 4(2\pi-\rho)\log \int_\Sph e^{u/2} + 2\int_\Sph u + C.
\end{align*}
As the functional $I$ is invariant under the addition of constants, we can assume that $\int_\S u = 0$ and apply Corollary \ref{WeakModified} and Proposition \ref{MosTrudBoundWeak} taking $a=2\rho$ and $b=4(2\pi-\rho)$, obtaining: 
\begin{align*}
I_\rho(u)\geq -2\rho \fint_\Sph u - 2(2\pi-\rho)\fint_\Sph u + 2 \int_\Sph u + C = C.
\end{align*}
We highlight that the constant $C$ does not depend on $\rho$.
\end{proof}

Proposition \ref{LowerBound} states that the functional $I$ is bounded from above, but we do not have coercivity. The reason for that is the non-compact action of the conformal group of the disk. This effect appears also in the Nirenberg problem in the sphere, for instance, and makes the problem rather difficult. 

We will show now that we can gain coercivity by restricting ourselves to spaces of symmetric functions. In order to do that, we introduce local versions of the inequalities above. These results are known as Chen-Li type inequalities (see \cite{chen-li} for more details).

\begin{proposition}\label{Local1} Let $\Sigma$ be a compact surface with $C^1$ boundary, $\Sigma_1\subset\Sigma$ and $\delta>0$ such that $(\Sigma_1)^{\delta}\cap\partial\Sigma = \emptyset.$ Then, for every $\e>0$, there exists a constant $C\in\R$ depending on $\e$ and $\delta$ such that $$16\pi \log \int_{\Sigma_1}e^u \leq \int_{(\Sigma_1)^\delta}\vert \nabla u \vert^2 + \e \int_\S \vert \nabla u\vert^2 + C, \hsp \forall u\in H^1(\Sigma) \mbox{ with } \int_\S u = 0.$$
\end{proposition}

 The details of the proof of this precise statement can be found in \cite[Proposition 2.2]{Ruiz}, for instance, but the idea dates back to \cite{chen-li}. Roughly speaking, one applies \eqref{equa3} to the function $u$ multiplied by a cut-off function in $\Sigma_1$.

If the function $u$ has mass in several separated regions satisfying the hypothesis of the propositions above, the obtained bounds improve by a factor of the number of such regions. This information is collected in the following corollary (see for instance \cite[Lemma 2.4]{Ruiz} for the case $l=2$; the case of general $l$ is analogous).

\begin{corollary} \label{Local2Corollary} Let $\Sigma$ be a compact surface with $C^1$ boundary, $l\in\N$ and $\Sigma_1, \ldots, \Sigma_l \subset \Sigma$ for which there exists a $\delta>0$ such that
	$(\Sigma_i)^\delta\cap(\Sigma_j)^\delta = \emptyset$ if $i\neq j$. Assume that there exists $\gamma \in (0,\frac{1}{l})$ such that $$\frac{\int_{\Sigma_i}e^u}{\int_\S e^u}\geq \gamma,\hsp\forall i=1,\ldots,l.$$Then, for every $\e>0$ there exists a constant $C\in\R$ depending on $\e$, $\delta$ and $\gamma$ such that 
	\begin{align*}
	8l\pi \log \int_\S e^u \leq \int_\S \vert \nabla u \vert^2 + \e \int_\S \vert \nabla u\vert^2 + C, \hsp \forall u\in H^1(\Sigma) \mbox{ with } \int_\S u = 0.
	\end{align*}
\end{corollary}

Using the same techniques we can give a localized version of the Proposition \ref{MosTrudBoundWeak}.

\begin{proposition}\label{Local3} Let $\S$ be a compact surface with $C^1$ boundary, and $\Gamma_1\subset \de\S$. Then, for every $\e,\delta>0$ there exists a constant $C\in\R$ depending on $\e$ and $\delta$ such that $$4\pi \log \int_{\Gamma_1}e^u \leq \int_{(\Gamma_1)^\delta}\vert \nabla u \vert ^2 + \e \int_\S \vert \nabla u \vert^2 + C, \hsp \forall u\in H^1(\Sigma) \mbox{ with } \int_\S u = 0$$
\end{proposition}

\begin{proof} Following \cite{chen-li}, we consider a cutoff function $g_\delta: \S \to [0,1]$ satisfying
\begin{equation*}
g_\delta=\left\lbrace \begin{array}{ll} 1 & \mbox{ if } x\in\Gamma_1, \\ 0 & \mbox{ si } x\in \Sigma\backslash (\Gamma_1)^{\delta/2}. \end{array} \right.
\end{equation*}
We have $g_\delta u\in H^1(\Sigma)$, hence we can apply Corollary \ref{WeakModified2}:
\begin{align}
4\pi \log \int_{\Gamma_1} e^u &= 4\pi \log \int_{\Gamma_1} e^{g_\delta u} \leq 4\pi \log \int_\deS e^{g_\delta u} \nonumber \\ &\leq \int_\S \vert \nabla(g_\delta u)\vert^2 + 4\pi\fint_\S g_\delta u + C. \label{equa6}
\end{align}
Then,
\begin{align}
\int_\S \vert \nabla(g_\delta u)\vert^2 &= \int_\S u^2 \vert \nabla g_\delta \vert^2 + 2 \int_\S g_\delta u \langle \nabla u, \nabla g_\delta \rangle + \int_\S (g_\delta)^2 \vert \nabla u \vert^2 \nonumber \\
&\leq C_\delta \int_\S u^2 + 2\int_\S g_\delta u \vert \nabla u \vert \vert \nabla g_\delta \vert + \int_{(\Gamma_1)^\delta} \vert \nabla u \vert^2. \label{equa7}
\end{align}
The central term can be bounded using Cauchy's inequality, obtaining
\begin{align}
\int_\S g_\delta u \vert \nabla u\vert \vert \nabla g_\delta \vert \leq C_\delta \int_\S u\vert \nabla u\vert \leq C_{\e,\delta} \int_\S u^2 + \e \int_\S \vert \nabla u \vert ^2. \label{equa8}
\end{align}
Combining \eqref{equa7} and \eqref{equa8}:
\begin{align}
\int_\S \vert \nabla(g_\delta u)\vert^2 \leq \int_{(\Gamma_1)^\delta}\vert \nabla u\vert^2 + \e \int_\S \vert \nabla u \vert^2 + C_{\e,\delta}\int_\S u^2. \label{equa9}
\end{align}
Also, we have the following bound for the mean value of $g_\delta u$ on $\de\S$:
\begin{equation}
\fint_\S g_\delta u \leq \fint_\S \frac{1}{2}((g_\delta)^2+u^2) \leq \frac{1}{2}\fint_\S (g_\delta)^2 + \frac{1}{2\vert \Sigma\vert}\int_\S u^2 \leq C_\delta + C\int_\S u^2. \label{equa10}
\end{equation}
Now, apply both inequalities \eqref{equa9} and \eqref{equa10} to \eqref{equa6} to get:
\begin{equation}
4\pi \log \int_{\Gamma_1}e^u \leq \int_{(\Gamma_1)^\delta} \vert \nabla u\vert^2 + \e \int_\S \vert \nabla u\vert^2 + C_{\e,\delta}\int_\S u^2 + C. \label{equa11}
\end{equation}
Finally we address the term $\int_\S u^2$.

Let $a\in\R$, $\eta = \vert \{x\in\Sigma: u(x)\geq a\}\vert$ and $(u-a)^+ = \max\:\{0,u-a\}$. Clearly, $u\leq (u-a)^+ + a$. We now apply formula \eqref{equa11} to the function $(u-a)^+$:
\begin{align}
4\pi \log \int_{\Gamma_1} e^u &\leq 4\pi \log \left(e^a \int_{\Gamma_1}e^{(u-a)^+} \right) \leq 4\pi a + \log \int_{\Gamma_1} e ^{(u-a)^+} \nonumber \\ &\leq 4\pi a + \int_{(\Gamma_1)^\delta} \vert \nabla(u-a)^+\vert^2 + \e \int_\S \vert \nabla(u-a)^+\vert^2 + C_{\e,\delta}\int_\S \left((u-a)^+\right)^2 \nonumber\\ &\leq 4\pi a + \int_{(\Gamma_1)^\delta}\vert \nabla u\vert^2 + \e \int_\S\vert \nabla u \vert ^2 + C_{\e,\delta}\int_\S \left((u-a)^+\right)^2. \label{equa12}
\end{align}
By means of Sobolev, H\"{o}lder and Poincar\'{e}-Wirtinger inequalities:
\begin{align} \label{equa13}
\int_\S \left((u-a)^+\right)^2 &= \int_{\{x\in\Sigma:\: u(x)\geq a\}} \left((u-a)^+ \right)^2 \leq \eta^{1/2} \left(\int_\S \left((u-a)^+\right)^4\right)^{1/2} \nonumber \\ &\leq \eta^{1/2} \Vert (u-a)^+\Vert^2_{H^1(\Sigma)} \leq C\eta^{1/2}\int_\S\vert \nabla u\vert^2.
\end{align}
Again by Poincar\'{e}-Wirtinger inequality:
\begin{align} \label{equa14}
a\eta \leq \int_{\{x\in\Sigma:\: u(x)\geq a\}} u \leq \int_\S \vert u \vert \leq C \left(\int_\S \vert u\vert^2\right)^{1/2} \leq C\left(\int_\S\vert \nabla u \vert^2\right)^{1/2}.
\end{align}
From \eqref{equa14}, using Cauchy's inequality:
\begin{align} \label{equa15}
a \leq \theta \int_\S \vert\nabla u\vert^2 + \frac{C^2}{\eta^2 \theta}, \hsp \forall \theta >0.
\end{align}
Mixing \eqref{equa12}, \eqref{equa13} and \eqref{equa15}:
\begin{align*}
4\pi\log\int_{\Gamma_1} e^u \leq 4\pi \theta \int_\S \vert\nabla u\vert^2 + \int_{(\Gamma_1)^\delta} \vert\nabla u\vert^2 + \e \int_\S \vert\nabla u \vert^2 + C_{\e,\delta}\eta^{1/2}\int_\S \vert \nabla u\vert^2 + C,
\end{align*}
and it is enough to take $\theta = \frac{1}{4\pi}$ and $\eta^{1/2}\leq \frac{\e}{C_{\e,\delta}}$ to conclude.
\end{proof}

\begin{corollary} \label{Local3Corollary} Let $\S$ be a compact surface with $C^1$ boundary, $l\in\N$ and $\Gamma_1, \ldots, \Gamma_l \subset \partial\Sigma$ for which there exists a $\delta>0$ such that $(\Gamma_i)^\delta\cap(\Gamma_j)^\delta = \emptyset$ if $i\neq j$. Moreover, assume that there exists $\gamma \in \:(0,\frac{1}{l})$ such that 

\begin{equation} \label{nueva} \frac{\int_{\Gamma_i}e^u}{\int_\deS e^u}\geq \gamma,\hsp \forall i=1,\ldots,l .\end{equation}

Then, for every $\e>0$ there exists a constant $C\in\R$, depending on $\e, \delta$ and $\gamma$, such that 
\begin{align*}
4l\pi \log \int_\deS e^u \leq \int_\S \vert \nabla u \vert^2 + \e \int_\S \vert \nabla u \vert^2 + C, \hsp \forall u\in H^1(\Sigma) \mbox{ with } \int_\S u = 0.
\end{align*}
\end{corollary}
\begin{proof}
First, we apply to each $\Gamma_i$ the previous result, obtaining
$$4\pi \log \int_{\Gamma_i} e^u \leq \int_{(\Gamma_i)^\delta}\vert \nabla u \vert^2 + \e \int_\S \vert \nabla u\vert^2 + C.$$
Using \eqref{nueva}, we obtain:
\begin{align*}
4\pi \log \int_{\Gamma_i} e^u \geq 4\pi \log \int_\deS e^u + C.
\end{align*}
Then,
\begin{align*}
4\pi \log \int_\deS e^u \leq \int_{(\Gamma_i)^\delta} \vert \nabla u\vert^2 + \e \int_\S\vert\nabla u\vert^2 + C.
\end{align*}
Finally, summing on $i\in\{1,\ldots,l \}$:
\begin{align*}
4l\pi \log \int_\deS e^u &\leq \int_{\bigsqcup_{i} (\Gamma_i)^\delta} \vert\nabla u\vert^2 + \e l \int_\S \vert\nabla u\vert^2 + C \\ &\leq \int_\S \vert \nabla u \vert^2 + \e l \int_\S \vert\nabla u\vert^2 + C.
\end{align*}
\end{proof}
We have just seen that the more regions the mass of a function is separated in, the better bounds we obtain using the local versions of the Moser-Trudinger inequalities. If a $H^1(\D)$ function is concentrated in an interior point of the disk, Proposition \ref{Local1} gives us a lower bound which is sufficient to achieve coercivity, but that is not the case when a function concentrates around a boundary point. To avoid this we will restrict ourselves to consider functions satisfying a symmetry condition guaranteeing that a function cannot concentrate around a single point of the boundary. Hence we will obtain coercivity by interpolating \ref{Local2Corollary} and \ref{Local3Corollary} with $l=2$.
\begin{figure}[h]
\includegraphics[scale=0.3]{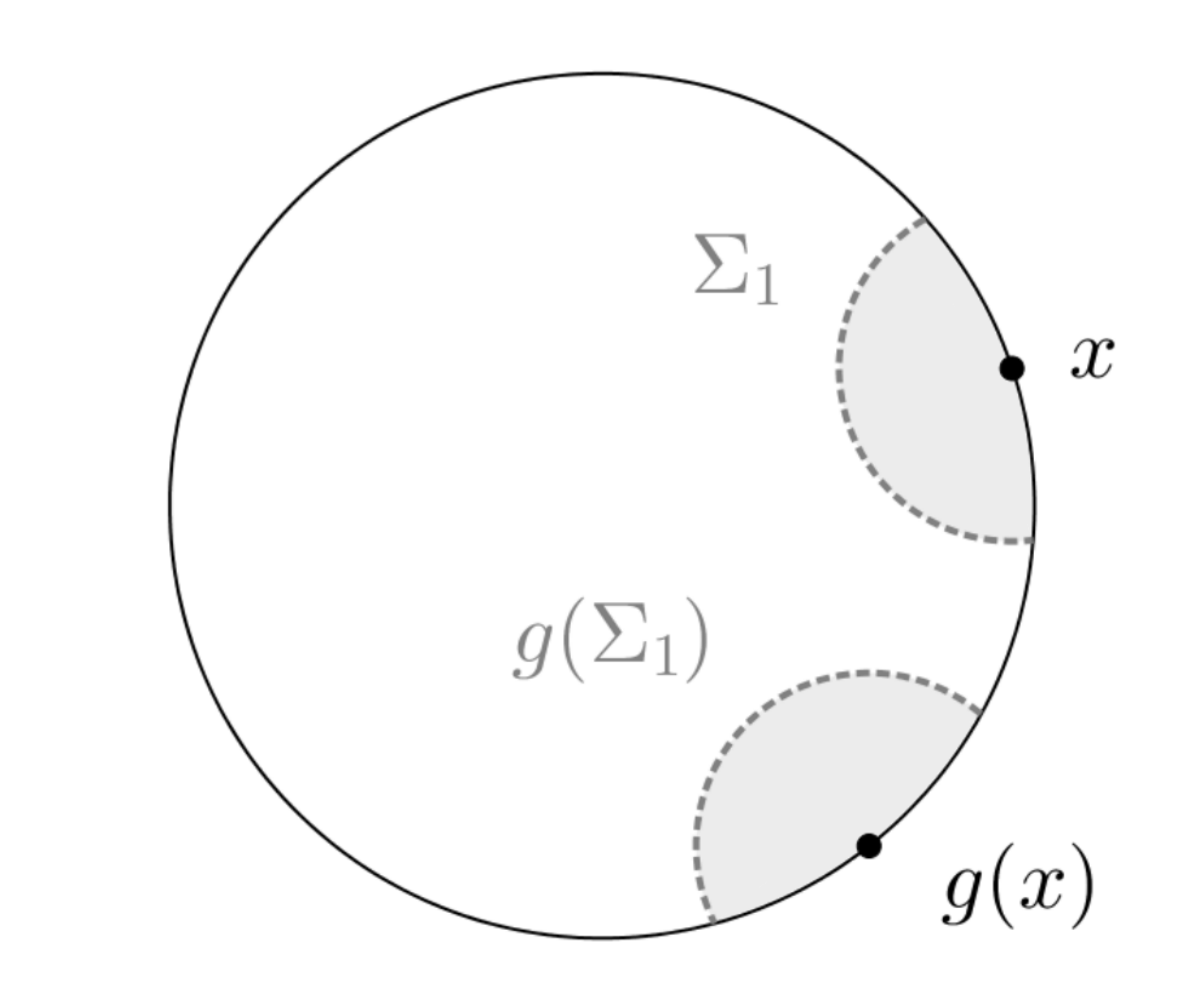}
\caption{If a symmetric function concentrates around $x\in\Sph$, then it also concentrates around $g(x)$ for all $g\in G$.}
\end{figure}

We let $G$ be a subgroup of the orthogonal transformation group of $\D$ such that the set of fixed points on $\Sph$ under the action of $G$ is empty, in other words,
\begin{equation*}
\{x\in\Sph: g(x)=x \hsp \forall g\in G\} = \emptyset.
\end{equation*}
For instance, we can take $G$ as the group of rotations generated by $g(z)= e^{\frac{2 \pi i}{k}}z$,  as well as the dihedral groups $\mathbb{D}_k$ ($k \in \mathbb{N}$, $k>1$).

In the sequel, $K$ and $h$ will be assumed to be $G-$symmetric functions, and we denote $H^1_G(\D)= \{ u \in H^1(\D):\ u \circ g = u \ \forall \ u \in G \}$, and: 
\begin{equation*}
\X_G = \{u\in \X: u\circ g = u \hsp \forall g\in G\}.
\end{equation*}

As in Lemma \ref{notempty} we observe that if $K$ and $h$ are G-symmetric functions somewhere positive, then $\X_G$ is not empty.

\begin{proposition}\label{Coercivity} Given $\rho\in \: [0,2\pi]$, the functional $I_\rho$ is coercive on $\X_G$, that is, $$I_\rho(u)\to +\infty \hsp (\Vert u\Vert_{H^1(\D)}\to +\infty, u\in \X_G).$$
\end{proposition}

\begin{proof} Take a sequence $(u_n)$ in $\X_G$. We know that $I_{\rho}$ is invariant under the addition of constants, so we can assume that $\int_\D u_n = 0$ for every $n\in\N$. We have
\begin{align*}
I_\rho(u_n)&\geq \frac{1}{2}\int_\D\vert \nabla u_n\vert^2 -2\rho \log \int_\D e^{u_n} -4(2\pi-\rho)\log \int_\Sph e^{u_n/2} + 2\int_\Sph u_n + C.
\end{align*}
Then, for any $a,b\in\R$ one has:
\begin{align*}
I_\rho(u_n) &\geq \frac{16\pi - 2a -b}{32\pi} \int_\D\vert \nabla u_n\vert^2 + \frac{a}{16\pi}\int_\D \vert \nabla u_n\vert^2 + \frac{b}{32\pi}\int_\D \vert \nabla u_n\vert^2 + 2\int_\Sph u_n \\ &-2\rho \log \int_\D e^{u_n} -4(2\pi -\rho)\log\int_\Sph e^{u_n/2}.
\end{align*}
We can now apply Corollaries \ref{Local2Corollary} and \ref{Local3Corollary} with $l=2$:
\begin{align*}
I_\rho(u_n)&\geq \frac{16\pi-2a-b}{32\pi}\int_\D \vert \nabla u_n\vert^2 + a\log\int_\D e^{u_n} - a \e \int_\D \vert \nabla u_n\vert^2 + b \log \int_\Sph e^{u_n/2} \\ &-b \e \int_\D \vert \nabla u_n\vert^2 - 2\rho \log \int_\D e^{u_n} - 4(2\pi-\rho)\log \int_\Sph e^{u_n/2} + 2 \int_\Sph u_n + C.
\end{align*}
Choosing $a=2\rho$ and $b=4(2\pi-\rho)$ and applying the trace inequality:
\begin{align*}
I_\rho(u_n) \geq \left(\frac{1}{4}-\e\right) \int_\D \vert \nabla u_n\vert^2 - 2C_2 \Vert u_n \Vert_{H^1(\D)} + C, \hsp C_2 >0.
\end{align*}
Finally, taking $\e$ small enough and using the Poincar\'{e}-Wirtinger inequality we obtain
\begin{equation}
I_\rho(u_n) \geq C_1 \Vert u_n \Vert_{H^1(\D)} ^2 - C_2 \Vert u_n \Vert_{H^1(\D)} + C, \hsp C_1, C_2 >0. \label{equa16}
\end{equation}
Again, we remark that the constant $C_1$ is independent of $\rho$.
\end{proof}

\section{Proof of Theorem \ref{TheoremNonNegative} and its generalization}
\label{sect3}
\setcounter{equation}{0} 

We begin this section considering the limiting cases $\rho=0$ and $\rho=2\pi$. These cases have their own interest, as will be shown, but their study will be useful also for the proof of Theorems \ref{TheoremNonNegative}, \ref{TheoremPerturbed} and \ref{TheoremGeneral}.

\medskip Observe that:

\begin{equation} \label{I(u,0)} I(u,0) = \frac{1}{2} \int_\D \vert \nabla u\vert^2  + 2\int_{\Sph} u - 8 \pi \log\int_{\Sph} he^{u/2} + 8\pi \log(2\pi),\end{equation}
and, as $K$ does not play any role, it can be defined on the bigger space
$$\X^1=\left\lbrace u\in H^1(\D): \int_\Sph he^{u/2} > 0 \right\rbrace \supset \X.$$
The critical points of $I_0$ on $\X^1$ are weak solutions of the problem
\begin{align*}
\left\lbrace \begin{array}{cc} -\Delta u = 0 & \mbox{ in } \D, \\[0.15cm] \frac{\partial u}{\partial \eta} + 2 = 4\pi \frac{he^{u/2}}{\int_\Sph he^{u/2}} & \mbox{ on } \Sph,
\end{array} \right.
\end{align*}
which is clearly equivalent to the problem of prescribing Gaussian curvature $K=0$ and geodesic curvature $h$, that is,

\begin{align}
\left\lbrace \begin{array}{cc} -\Delta u = 0 & \mbox{ in } \D, \\[0.15cm] \frac{\partial u}{\partial \eta} + 2 = 2h e ^{u/2} & \mbox{ on } \Sph.
\end{array} \right.\label{Probl4}
\end{align}

Under the hypothesis that $h$ is $G$-symmetric, we can seek a minimizer of $I_0$ on the space of symmetric functions $$\X^1_G = \left\lbrace u\in \X^1:u\circ g = u \hsp \forall g\in G \right\rbrace.$$

\begin{theorem} \label{TheoremGeodesic} Let $h: \Sph \to \R$ be a H\"{o}lder continuous, somewhere positive $G-$symmetric function. Then Problem \eqref{Probl4} admits a solution as a minimum of $I_0$ on $\X^1_G$.
\end{theorem}

\begin{proof} The functional is bounded from below as seen in Proposition \ref{LowerBound}, so there exists $$\alpha = \inf_{u\in\X^1_G}I_0(u).$$
Let $(u_n)$ be a minimizing sequence in $\X^1_G$, that is, $I_0(u_n)\to \alpha$. By Proposition \ref{Coercivity} we know that $I_0$ is coercive so $u_n$ is bounded in the $H^1(\D)$ norm and we can assume that, there exists $u_0$ in $H^1(\D)$ such that, up to a subsequence, $u_n\weakto u_0$. Then, we also have $$\int_\Sph u_n \to \int_\Sph u_0,\hsp \int_\Sph he^{u_n/2} \to \int_\Sph he^{u_0/2}.$$
Combining this information with the fact that the function $u\to \int_\D \modgrad{u}$ is weakly lower semicontinuous, we have $I_0(u_0)\leq \alpha$. It is easy to check that $\int_\Sph he^{\frac{u_0}{2}}>0$, because if we had $\int_\Sph he^{u_n/2}\to 0$ then $I_0(u_n)\to +\infty$, which contradicts that $u_n$ is minimizing. Also, notice that weak convergence respect symmetry, so $u_0$ is a $G-$symmetric function.
\end{proof}

Analogously, we can consider the functional related to the limiting case $\rho= 2\pi$, 

 \begin{equation} \label{I(u, 2pi)} I(u, 2\pi)= \frac{1}{2} \int_\S \modgrad{u} + 2\int_\Sph u -4\pi \log \int_\D K e^u + 4\pi +4\pi\log(2\pi). \end{equation}
defined on the space
$$\X_G^2 = \left\{u\in H^1_G(\D): \int_\D Ke^u >0 \right\} \supset \X_G.$$
One can check that its variation with respect to $u$ produces weak solutions of the problem
\begin{align*}
\left\lbrace \begin{array}{cc} -\Delta u = 4\pi\frac{Ke^u}{\int_\D Ke^u} & \mbox{ in } \D, \\[0.15cm] \frac{\partial u}{\partial \eta} + 2 = 0 & \mbox{ on } \Sph,
\end{array} \right.
\end{align*}
which is equivalent to the problem of prescribing geodesic curvature $h=0$ and Gaussian curvature $K$:
\begin{align}
\left\lbrace \begin{array}{cc} -\Delta u = 2Ke^u & \mbox{ in } \D, \\[0.15cm] \frac{\partial u}{\partial \eta} + 2 = 0 & \mbox{ on } \Sph.
\end{array} \right. \label{Probl5}
\end{align}

A trivial adaptation of the proof of Theorem \ref{TheoremGeodesic} gives the following:

\begin{theorem} \label{TheoremGaussian} Let $K: \D \to \R$ be a H\"{o}lder continuous, somewhere positive $G-$symmetric function. Then Problem \eqref{Probl4} admits a solution as a minimum of $I_{2\pi}$ on $\X^2_G$.
\end{theorem}

\begin{remark} The existence result of Theorem \ref{TheoremGeodesic} is known, see for instance \cite{liuwang}. We have not found a explicit statement of the existence result of Theorem \ref{TheoremGaussian}, but we guess that it must be also known. However in this section we have reinterpreted those solutions as minimizers of $I_0$ and $I_{2\pi}$, respectively. This will be of use in what follows.

\end{remark}

%We consider $K$ and $h$ continuous, non-negative and $G-$symmetric functions. If one of the curvatures is identically null, then the equation 
%$$\int_\D K e^u + \int_\Sph he^{\frac{u}{2}} = 2\pi$$
%tells us that the other curvature must be positive somewhere, so we are in one of the limiting cases previously treated. Then, we can assume that both curvature functions are positive in some point.
%	
%In this case, the space $\X$ coincides with $H^1(\D)$, and the coercivity of each functional $I_{\rho}$ in $\X_G$ implies the existence of a minimum depending on $\rho$. In order to get a global minimum for $I$ we must verify that the function $\rho\to \min_{\X_G}I_\rho$ does not achieve its minimum in the boundary values $\{0,2\pi\}$.
%	

Let us now conclude the proof of Theorem \ref{TheoremNonNegative}.

\begin{proof}[Proof of Theorem \ref{TheoremNonNegative}.]

If $K =0$ or $h = 0$, then we are under the assumptions of  Theorem \ref{TheoremGeodesic} or Theorem \ref{TheoremGaussian}. Then, we can assume that both $K$ and $h$ are positive in some point and non-negative. In this case, $\X_G = \X^1_G =\X^2_G = H^1_G(\D)$.

By Proposition \ref{Coercivity}, there exists $(\hat{u}, \hat{\rho}) \in H^1_G(\D) \times [0,2\pi]$ a minimizer for $I$. We conclude if we exclude the possibilities $\hat{\rho}=0$ or $\hat{\rho}=2\pi$. 

Assume that $\hat{\rho}=0$. Observe that in this case, $\hat{u}$ is a minimizer for $I(\cdot, 0)$. Then,

\begin{align*} I(\hat{u},0) \leq I(\hat{u}, \rho) = I(\hat{u},0)  -2 \rho \log \Big (\int_{\D} K e^{\hat{u}} \Big ) + 4 \rho \log \Big (\int_{\Sph} h e^{\hat{u}/2} \Big ) \\ + 8 \pi \log \Big (\frac{2\pi-\rho}{2\pi} \Big)- 4 \rho \log (2 \pi- \rho) + 2 \rho + 2 \rho \log \rho.\end{align*}
But observe that, as $\rho \to 0$, the main term above is $2 \rho \log \rho$, which is negative. This gives a contradiction that excludes the case $\hat{\rho}=0$. One can exclude the case $\hat{\rho} = 2 \pi$ in an analogous way.

\end{proof}

The proof of Theorem \ref{TheoremNonNegative} can be adapted to a more general setting as follows:
\begin{theorem}\label{TheoremGeneral} Let $K$ and $h$ be H\"{o}lder continuous $G-$symmetric functions that are positive somewhere. We define
\begin{align*}
S_0 = \left\lbrace u\in \X^1_G:\: I_0(u)=\min_{\X^1_G} I_0  \right\rbrace, \hsp S_{2\pi} = \left\lbrace u\in \X^2_G:\: I_{2\pi}(u)=\min_{\X^2_G} I_{2\pi} \right\rbrace.
\end{align*}
If $S_0\cap\X_G$ and $S_{2\pi}\cap\X_G$ are nonempty, then \eqref{Probl} admits a solution.
\end{theorem}

Clearly, Theorem \ref{TheoremNonNegative} is an immediate consequence of Theorem \ref{TheoremGeneral}. Notice also that the sets $S_0$ and $S_{2\pi}$ of the hypotheses are nonempty because of Theorems \ref{TheoremGeodesic} and \ref{TheoremGaussian}.

\begin{proof}
The proof follows the same energy comparison argument than above, but a couple of details are worth to be written down. First, the existence of a minimizer is not clear a priori. Then let $(u_n, \rho_n ) \in \X_G \times (0,2\pi)$ be a minimizing sequence, that is, $I(u_n, \rho_n) \to \inf I$. Clearly $u_n$ is bounded in $H^1_G(\D)$ by Proposition \ref{Coercivity}, but its weak limit $\hat{u}$ could fall outside $\X_G$.

If $\rho_n \to \hat{\rho} \in (0,2\pi)$, from the fact that $I(u_n, \rho_n)$ is bounded we obtain:

$$ 0< \e < \int_{\D} K e^{u_n} < C, \ \ 0< \e < \int_{\Sph} h e^{u_n/2}< C,$$
for some $\e>0$, $C>0$. As a consequence $u_n \rightharpoonup \hat{u} \in \X_G$ and we are done. 

\medskip Assume now that $\rho_n \to 0$. If $n$ is sufficiently large we have the estimate:

$$I(u_n,\rho_n)\geq -2\rho_n \log\left(\int_\D Ke^{u_n}\right)-4(2\pi-\rho_n)\log\left(\int_\Sph he^{u_n/2}\right)+C.$$
Notice that $$\liminf\limits_{n\to\infty}-2\rho_n \log\left(\int_\D Ke^{u_n} \right)\geq 0.$$ Thus, $-\log\left(\int_\Sph h e^{u_n/2}\right)$ must be bounded from above, which means that $$0<\e<\int_\Sph he^{u_n/2}.$$
%Again by energy estimates, 
%
%$$ 0< \e < \int_{\Sph} h e^{u_n/2}<M,$$
%
%but $\int_\D K e^{u_n}$ could converge to $0$ a priori. In any case we can write:
Now, we write:
\begin{align*}I(u_n,\rho_n) = I(u_n,0)  -2 \rho_n \log \Big (\int_{\D} K e^{u_n} \Big ) + 4 \rho_n \log \Big (\int_{\Sph} h e^{u_n/2} \Big ) \\+ 8 \pi \log \Big (\frac{2\pi-\rho_n}{2\pi} \Big)- 4 \rho_n \log (2 \pi- \rho_n) + 2 \rho_n + 2 \rho_n \log \rho_n.\end{align*}
From this we deduce that:

$$ \inf I = \lim_{n \to \infty} I(u_n,\rho_n) \geq \liminf_{n \to \infty}  I(u_n,0) \geq I(u_0,0),$$
where $u_0 \in S_0 \cap \X_G$. But, as in the proof of Theorem \ref{TheoremNonNegative},

$$I(u_0,0) > I(u_0, \rho),$$
for small values of $\rho$. This contradiction shows that $\rho_n $ cannot converge to $0$. In an analogous way we can exclude its convergence to $2 \pi$.

\end{proof}

\section{A perturbation result}
\label{sect4}
\setcounter{equation}{0}

In this section it is necessary to specify the dependence of $I$ on the curvature functions $K$ and $h$, so we are writting $I(u, \rho) = I[K,h](u,\rho)$. We begin with a compactness result:

\begin{lemma}\label{Compactness} Let $(K_n)$ and $(h_n)$ be sequences of  H\"{o}lder continuous $G-$symmetric functions, defined on $\D$ and $\Sph$ respectively, such that
\begin{align*}
K_n &\to K \hsp \mbox{uniformly in } \D \mbox{ and } K\in C^{0,\alpha}(\D), \\
h_n &\to h \hsp \mbox{uniformly on } \Sph \mbox{ and } h\in C^{0,\alpha}(\Sph).
\end{align*}
Let us consider a sequence $(u_n)$, where each $u_n$ is a solution of the problem 
\begin{align} \label{Probl6}
\left\lbrace \begin{array}{ll} -\Delta u = 2K_n e^u & \mbox{in } \D, \\ \frac{\partial u}{\partial n}+2 = 2h_ne^{u/2} & \mbox{on } \Sph, \end{array}\right.
\end{align}
satisfying
\begin{align}
\rho_n = \int_\D K_n e^{u_n} > 0,\hsp \int_\Sph h_n e^{u_n/2} > 0, \hsp \forall n\in\N. \label{equa17}
\end{align}
Assume that $I[K_n,h_n](u_n, \rho_n)$ is uniformly bounded from above. Then $u_n\weakto u_\infty$ on $H^1(\D),$ being $u_\infty$ a solution of the problem 
\begin{align}
\left\lbrace \begin{array}{ll} -\Delta u = 2K e^u & \mbox{in } \D, \\ \frac{\partial u}{\partial n}+2 = 2he^{u/2} & \mbox{on } \Sph. \end{array}\right. \label{Probl7}
\end{align}
\end{lemma}
\begin{proof} First, we notice that  $\Vert K_n - K\Vert_\infty \to 0$ and $\Vert h_n - h\Vert_\infty \to 0$ imply that for every $\e > 0$ there exists $n_0\in\N$ such that, for $n\geq n_0$: $$\Vert K_n \Vert_\infty < \Vert K \Vert_\infty+\e,\hsp \Vert h_n \Vert_\infty < \Vert h \Vert_\infty+\e.$$
Hypothesis \eqref{equa17} gives us $0< \rho_n < 2\pi$ for all $n\in\N$. Then, for $n\geq n_0$ we have the following bound:
\begin{equation*}
I[K_n, h_n](u_n,\rho_n) \geq I(\Vert K\Vert_\infty + \e, \Vert h\Vert_\infty + \e)(u_n,\rho_n) .
\end{equation*}
And then, by Proposition \ref{Coercivity}, there exist constants $C_1, C_2 >0$, independent of $n$, such that 
\begin{align*}
I[K_n, h_n](u_n,\rho_n)\geq C_1\Vert u_n\Vert_{H^1}^2-C_2\Vert u_n\Vert_{H^1}+C_\e.
\end{align*}
Taking into account the hypothesis that $I[K_n, h_n](u_n,\rho_n)$ is uniformly bounded from above we have immediately that $u_n$ is bounded in the $H^1(\D)$ norm. Hence, up to a subsequence we can assume that there exists $u_\infty\in H^1(\D)$ such that $u_n\weakto u_\infty$.
	
Then, it is known that $2K_ne^{u_n}\to 2Ke^{u_\infty}$ and $2h_n e^{u_n/2} \to 2he^{\frac{u_\infty}{2}}$ on $L^p$ for $1\leq p < +\infty$, and that $\langle \nabla u_n, w \rangle\to \langle \nabla u_\infty, w\rangle$ for all $w\in H^1(\D)$. In particular $$ u_n|_{\Sph} \to  u_\infty|_{\Sph} \ \mbox{ in } L^2(\Sph).$$
We now pass to the limit in the weak formulation of \eqref{Probl6}:
\begin{align}
\int_\D \langle\nabla u_n ,\nabla v\rangle -2  \int_\D K_n e^{u_n} v +2 \int_\Sph v -\int_\Sph h_ne^{u_n/2} v = 0. \label{WeakSol}
\end{align}
for all $v\in H^1(\D)$. As a consequence $u_{\infty}$ is a weak solution of \eqref{Probl7}. By standard regularity estimates $u_\infty$ is indeed a classical solution.

\end{proof}
%\begin{remark} If either $K_n$ or $h_n$ are identically zero, then the statement is still valid setting $\rho_n = 0$ or $\rho_n = 2\pi$ respectively and removing in \eqref{equa17} the condition relative to the null sequence.
%\end{remark}
The next step is to check that, when considering a sequence of minimum type solutions, the hypothesis of Lemma \ref{Compactness} are automatically satisfied. Observe that under our hypotheses $I[K_n,h_n](\cdot,\cdot)\to I[K,h](\cdot,\cdot)$ pointwise in $\X\times (0,2\pi)$. 
%\begin{align*}
%\left|\int_\D K_ne^u - \int_\D Ke^u\right| &\leq \Vert K_n-K\Vert_\infty \int_\D e^u \to 0,\hsp \forall u\in\X \\ \left|\int_\Sph h_ne^{\frac{u}{2}} - \int_\Sph he^{\frac{u}{2}}\right| &\leq \Vert h_n-h\Vert_\infty \int_\Sph e^{\frac{u}{2}} \to 0,\hsp \forall u\in\X
%\end{align*}
%and by continuity of the $\log$
%\begin{align*}
%&\left|I(K_n,h_n)(u,\rho)-I(K,h)(u,\rho) \right| \leq 2\rho\left|\log \int_\D K_n e^u - \log \int_\D K e^u \right|+ \\ &+ 4(2\pi-\rho)\left| \log\int_\Sph h_n e^{\frac{u}{2}} - \log\int_\Sph he^{\frac{u}{2}} \right|\to 0,\hsp \forall (u,\rho)\in \X\times(0,2\pi)
%\end{align*}

Then, if $(u_n,\rho_n)$ is a sequence of minimum type solutions of \eqref{Probl6}, $$\limsup_{n\to+\infty}I[K_n,h_n](u_n,\rho_n)= \limsup_{n\to+\infty}\min_{\X\times(0,2\pi)}I[K_n,h_n](\cdot,\cdot)\leq \min_{\X\times(0,2\pi)}I.$$
Where the previous inequality is due to the fact that $(f_n)$ converging pointwise to $f$ implies
$\lim_{n\to+\infty}\inf f_n(y)\leq \inf f(y)$.
\begin{proof}[Proof of Theorem \ref{TheoremPerturbed}] We apply Theorem \ref{TheoremGeneral} to the problems
\begin{align*}
\left\lbrace \begin{array}{ll} -\Delta u = 2K e^u & \mbox{in } \D, \\[0.15cm] \frac{\partial u}{\partial \eta}+2 = 2he^{u/2} & \mbox{on } \Sph \end{array}\right.
\end{align*}
for which we need that the limiting problems
\begin{align*}
(P^1_{K})\left\lbrace \begin{array}{ll} -\Delta u = 2K e^u & \mbox{in } \D \\ \frac{\partial u}{\partial n}+2 = 0 & \mbox{on } \Sph \end{array}\right.,\hsp (P^2_{h})\left\lbrace \begin{array}{ll} -\Delta u = 0 & \mbox{in } \D \\ \frac{\partial u}{\partial \eta}+2 = 2he^{u/2} & \mbox{on } \Sph \end{array}\right.
\end{align*}
admit minimum type solutions, $u_1$ and $u_2$ respectively, verifying 
$$\int_\D Ke^{u_2}>0,\hsp \int_\Sph he^{u_1/2}>0.$$ 

By contradiction, take $K_n$ and $h_n$  H\"{o}lder continuous functions converging uniformly to $K_0$ and $h_0$. We can assume that $n$ is large enough so that $K_n$ and $h_n$ are somewhere positive, so that solutions for the limiting problems in the form of minimizers can be found via Theorems \ref{TheoremGeodesic} and \ref{TheoremGaussian}. Now, take $(\tilde{u}_n)$ a sequence of minimum type solutions of the problems $(P^1_{K_n})$ and $(\hat{u}_n)$ a sequence of minimum type solutions of the problems $(P^2_{h_n})$ such that
\begin{equation}\label{equa18}
\mbox{either } \int_\D K_n e^{\hat{u}_n}\leq 0, \ \mbox{ or }\ \int_\Sph h_ne^{\tilde{u}_{n}/2}\leq 0, \hsp \forall n\in \N.
\end{equation}
By Lemma \ref{Compactness} we know that $\hat{u}_n\weakto\hat{u}$ and $\tilde{u}_n\weakto\tilde{u}$, solutions for the limiting problems $(P^1_{K_0})$ and $(P^2_{h_0})$. Taking limit when $n\to +\infty$ in \eqref{equa18} we obtain: 

$$ \mbox{either }  \int_\D K_0\,e^{\tilde{u}}\leq 0, \ \mbox{ or }\ \int_\Sph h_0\,e^{\hat{u}/2}\leq 0,$$ which is a contradiction since both $K_0$ and $h_0$ are nonnegative functions somewhere positive.
\end{proof}

\end{document}